\documentclass[11pt]{amsart}
\usepackage{amssymb,amsmath,amsthm,amsfonts,amsopn,url,color,hyperref,enumitem,tikz-cd}
\usepackage[all]{xy}
\usetikzlibrary{decorations.markings}
\tikzset{crossout/.style={
        decoration={markings,
            mark= at position 0.5 with {
                \node[transform shape] (tempnode) {$\backslash$};}},
        postaction={decorate}}}
\newcommand{\tikst}[1]{\begin{tabular}{@{}c@{}}#1\end{tabular}}

\theoremstyle{plain}
\newtheorem{thm}{Theorem}[section]
\newtheorem{prop}[thm]{Proposition}
\newtheorem{lemma}[thm]{Lemma}
\newtheorem{cor}[thm]{Corollary}

\theoremstyle{definition}
\newtheorem{defn}[thm]{Definition}
\newtheorem*{defn*}{Definition}
\newtheorem*{question*}{Question}
\newtheorem{question}{Question}
\newtheorem{example}[thm]{Example}
\newtheorem*{example*}{Example}

\newtheorem*{rmk*}{Remark}
\newcommand{\field}[1]{\mathbb{#1}}
\newcommand{\N}{\field{N}}
\newcommand{\Z}{\field{Z}}
\newcommand{\Q}{\field{Q}}
\newcommand{\R}{\field{R}}
\newcommand{\C}{\field{C}}
\newcommand{\A}{\field{A}}
\newcommand{\ideal}[1]{\mathfrak{#1}}
\newcommand{\m}{\ideal{m}}
\newcommand{\n}{\ideal{n}}
\newcommand{\p}{\ideal{p}}
\newcommand{\q}{\ideal{q}}

\newcommand{\func}[1]{\mathrm{#1} \,}

\newcommand{\Spec}{\func{Spec}}

\newcommand{\hgt}{\func{ht}}

\newcommand{\arrow}[1]{\stackrel{#1}{\rightarrow}}

\newcommand{\li}
 {\leftfootline}

\newcommand{\cO}{\mathcal{O}}
\renewcommand{\phi}{\varphi}

\DeclareMathOperator{\Cl}{Cl}
\DeclareMathOperator{\chr}{char}
\DeclareMathOperator{\Max}{Max}
\newcommand{\Rone}{(R$_1$)}
\newcommand{\Ronei}{$($\emph{R}$_1)$}

\author{Neil Epstein}
\address{Department of Mathematical Sciences \\ George Mason University \\ Fairfax, VA  22030}
\email{nepstei2@gmu.edu}

\author{Jay Shapiro}
\address{Department of Mathematical Sciences \\ George Mason University \\ Fairfax, VA  22030}
\email{jshapiro@gmu.edu}

\title{Perinormality -- a generalization of Krull domains}

\date{\today}
\begin{document}
\begin{abstract}
We introduce a new class of integral domains, the \emph{perinormal} domains, which fall strictly between Krull domains and weakly normal domains.  We establish basic properties of the class, and in the case of universally catenary domains we give equivalent characterizations of perinormality.  (Later on, we point out some subtleties that occur only in the non-Noetherian context.)  We also introduce and explore briefly the related concept of \emph{global} perinormality, including a relationship with divisor class groups.  Throughout, we provide illuminating examples from algebra, geometry, and number theory.
\end{abstract}

\maketitle

\section{Introduction}\label{sec:intro}
Motivated in part by the classical concept of a ring extension satisfying going-down from Cohen and Seidenberg \cite{CohSei}, the concept of the \emph{going-down domain}  has been fruitful in non-Noetherian commutative ring theory (see for example \cite{Dob-gd, Dob-gd2, DobPa-gd3}); for Noetherian rings it merely coincides with domains of dimension $\leq 1$ \cite[Proposition 7]{Dob-gd}.    By definition, a ring extension $R \subseteq S$ satisfies \emph{going-down} if  whenever $\p \subset \q$ are prime ideals of $R$ and $Q \in \Spec S$ with $Q \cap R = \q$, there is some prime ideal $P \in \Spec S$ with $P \subset Q$ and $P \cap R = \p$ (a condition that is satisfied whenever $S$ is flat over $R$).
Then an integral domain $R$ is a \emph{going-down domain} if for every (local) overring $S$ of $R$, the inclusion $R \subseteq S$ satisfies going-down. (In fact by \cite[Theorem 1]{DobPa-gd3}  it doesn't matter whether one specifies `local' or not.)

It is natural to ask which overrings of an integral domain $R$ satisfy going-down over it.  It is classical that any flat $R$-algebra (hence any flat overring) will satisfy going-down over $R$ \cite[Theorem 9.5]{Mats}.  Moreover, the flat local overrings are precisely the rings $R_\p$ where $\p$ is a prime ideal of $R$ \cite[Theorem 2]{Ri-gq}.  In this context, since going-down domains have proven to be a useful concept, it makes sense  to explore the orthogonal concept: \begin{itemize}
\item When does it happen that the \emph{only} local overrings that satisfy going-down over $R$ are the localizations at prime ideals?
\end{itemize} We call such a ring \emph{perinormal}, and it is the subject of this paper.  (The related concept of \emph{global perinormality} stipulates that the only overrings, local or not, that satisfy going-down over the base are localizations of the base ring at multiplicative sets.)

  It turns out that the class of perinormal rings is closely related to  Krull domains (and so Noetherian normal domains) and   weakly normal (hence seminormal) domains in that Krull domain $\Rightarrow$ perinormal $\Rightarrow$ weakly normal
   and \Rone, with neither implication reversible.  Moreover, for universally catenary domains (and somewhat more generally), we can characterize perinormal domains as those domains $R$ such that no prime localization $R_\p$ has an overring that induces a bijection on prime spectra.  When $R$ is smooth in codimension 1, we can restrict our attention to \emph{integral} overrings of these $R_\p$ (cf. Theorem~\ref{thm:equiv}).  On the other hand, the only perinormal going-down domains are Pr\"ufer domains.

The structure of the paper is as follows.  We start by establishing some basic facts in Section~\ref{sec:prelim}, including Proposition~\ref{pr:perilocal} which shows that perinormality is a local property and Proposition~\ref{pr:periflat}, which reframes perinormality in terms of flatness.  Section~\ref{sec:R1Krull} explores the relationship of perinormality to (generalized) Krull domains, weakly normal domains, and \Rone\ domains.  Theorem~\ref{thm:Krperi}, Corollary~\ref{cor:pnwn}, and  Proposition~\ref{pr:R1} respectively show that perinormality is implied by the first and implies the latter two properties.  We also exhibit some sharpening examples.  Section~\ref{sec:locNoeth} is dedicated to Theorem~\ref{thm:equiv}, which gives the two characterizations of perinormal domains among the Noetherian domains mentioned above. In Section~\ref{sec:gluing}, we find Theorem~\ref{thm:gluing}, which exhibits a method for producing perinormal domains that are not integrally closed.  Section~\ref{sec:global} is devoted to the related notion of \emph{global} perinormality; in particular, we give a partial characterization (see Theorem~\ref{thm:class}) of which Krull domains may be globally perinormal, along with examples relevant to algebraic number theory.  It turns out that the theory of perinormality is a bit different when one includes non-Noetherian rings; in Section~\ref{sec:nonNoeth}, we point out the subtleties in a series of examples, including the fact that not every integrally closed domain is perinormal (unlike in the Noetherian case).  We end with a list of interesting questions in Section~\ref{sec:Q}.

\noindent \textbf{Conventions:} All rings are commutative with identity, and ring homomorphisms and containments preserve the multiplicative identity.  The term \emph{local} means only that the ring has a unique maximal ideal.  An \emph{overring} of an integral domain $R$ is a ring sitting between $R$ and its fraction field.

\section{First properties}\label{sec:prelim}

\begin{defn}
Let $R$ be an integral domain.  We say $R$ is \emph{perinormal} if whenever $S$ is a local overring of $R$ such that the inclusion $R \subseteq S$ satisfies going-down, it follows that $S$ is a localization of $R$ (necessarily at a prime ideal).

We say that $R$ is \emph{globally perinormal} if the same conclusion holds when the condition on $S$ being local is dropped (so that this time, the localization is just at a multiplicative set).
\end{defn}

\begin{rmk*}
The term \emph{perinormal} is meant to reflect several aspects of the property:  (1) It is closely related to the properties of seminormality and weak normality (cf. Corollary~\ref{cor:pnwn}).  (2) It is closely related to the concept of normality in the Noetherian case (cf. Theorem~\ref{thm:Krperi}), which is usually the only situation where the word ``normal'' is used for integral closedness in the literature.  (3) Perinormality is \emph{not} a weakening of the property of integral closedness in general (cf. Example~\ref{ex:Hutchins}), whence the prefix \emph{peri-} (unlike \emph{weak} and \emph{semi-}, which both imply weakenings).
\end{rmk*}

\begin{lemma}\label{lem:surj}
A homomorphism of commutative rings $R \rightarrow S$ satisfies going-down if and only if for all $P \in \Spec S$, the induced map $\Spec (S_P) \rightarrow \Spec (R_{P \cap R})$ is surjective.
\end{lemma}

\begin{proof}
This follows immediately from the definition.
\end{proof}

Recall the following theorem of Richman.
\begin{thm}[{\cite[Theorem 2]{Ri-gq}}]\label{thm:Rich}
Let $A$ be an integral domain and $B$ an overring of $A$. Then $B$ is flat over $A$ if and only if $B_\m = A_{\m\cap A}$ for all maximal ideals $\m$ of $B$.
\end{thm}

This theorem allows us to characterize perinormality in terms of flatness:

\begin{prop}\label{pr:periflat}
A domain $R$ is perinormal if and only if every overring of $R$ that satisfies going-down is flat over $R$.
\end{prop}

\begin{proof}
Suppose $R$ is perinormal.  Let $S$ be an overring of $R$ that satisfies going-down over $R$.  Let $\m$ be a maximal ideal of $R$.  Clearly $S_\m$ satisfies going-down over $R$ (since going-downness is transitive and $S_\m$ is going-down over $S$), so by perinormality, $S_\m$ is a localization of $R$, so that necessarily, $S_\m = R_{\m \cap R}$.  Then by Theorem~\ref{thm:Rich}, $S$ is flat over $R$.

Conversely, suppose every overring of $R$ that satisfies going-down is flat over $R$.  Let $(S,\m)$ be a local overring of $R$ that satisfies going-down.  Then it is flat over $R$, so again by Theorem~\ref{thm:Rich}, $S = R_{\m \cap R}$ is a localization of $R$.  Thus, $R$ is perinormal.
\end{proof}

Next, we show that perinormality is a local property.

\begin{prop}\label{pr:perilocal}
If $R$ is perinormal, so is $R_W$ for every multiplicative set $W$.  Conversely, if $R_\m$ is perinormal for all maximal ideals $\m$ of $R$, then so is $R$.
\end{prop}

\begin{proof}
First suppose $R$ is perinormal.  Let $S$ be a local overring of $R_W$ that satisfies going-down.   Then $S$ satisfies going-down over $R$ (since no prime ideal of $R$ lain over by a prime of $S$ can intersect $W$) and so $S=R_V$ for some multiplicatively closed subset $V$ of $R$.  But then $V$ is also a multiplicatively closed subset of $R_W$, and $S=(R_W)_V$.  Therefore $R_W$ is perinormal.

Conversely, suppose that $R_\m$ is perinormal for all maximal ideals $\m$ of $R$.  Let $(S,\n)$ be a local overring of $R$ such that the inclusion $R \subseteq S$ satisfies going-down.  Let $\m$ be a maximal ideal of $R$ such that $\n \cap R \subseteq \m$.  Then $R_\m \subseteq S$ satisfies going-down, so that by perinormality of $R_\m$, $S=(R_\m)_{\n \cap R_\m} = R_{\n \cap R}$.  Thus $R$ is perinormal.
\end{proof}

\begin{example}
Any valuation domain $R$ is globally perinormal because \emph{every} overring of $R$ is a localization, as is easily shown.  It then follows from Proposition~\ref{pr:perilocal} that every Pr\"ufer domain is perinormal.
\end{example}

\section{\Rone\ domains, weakly normal domains, and generalized Krull domains}\label{sec:R1Krull}

In this section, we fit perinormality into the context of three known important classes of integral domains.  Namely, generalized Krull $\implies$ Krull $\implies$ perinormal $\implies$ weakly normal and \Rone, with neither arrow reversible. 

\begin{defn}
We say that a commutative ring $R$ satisfies \Rone\ if $R_P$ is a valuation domain whenever $P$ is a height one prime of $R$.
\end{defn}

\begin{rmk*}
It seems that in the literature, the term \Rone\ is only used for Noetherian rings (cf. \cite[p.\! 183]{Mats}).  Here we have extended it to arbitrary commutative rings in a way that both coincides with the established definition in the Noetherian case and suits our purpose in the general case.
\end{rmk*}

\begin{prop}\label{pr:R1}
Any perinormal domain $R$ satisfies \Ronei.
\end{prop}

\begin{proof}
Let $\p$ be a height one prime of $R$.  Let $(V,\m)$ be a valuation overring of $R$ such that $\m \cap R = \p$. (If $R$ is Noetherian, we can choose $V$ to be Noetherian as well.)  Then the map $R \rightarrow V$ trivially satisfies going-down.  Thus, $V$ is a localization of $R$, whence $V = R_{\m \cap R} = R_\p$, completing the proof that $R$ satisfies \Rone.
\end{proof}

\begin{prop}\label{pr:intbijection}
If $(R,\m)$ is a local perinormal domain, then for any integral overring $S$ of $R$ such that $\Spec S \rightarrow \Spec R$ is a bijection, we have $R=S$.
\end{prop}

\begin{proof}
Let $S$ be an integral overring of $R$ such that $\Spec S \rightarrow \Spec R$ is bijective.  By integrality of the extension, some prime ideal $\n$ of $S$ lies over $\m$; by bijectivity, there can be only one such prime; since fibers of Spec maps on integral extensions are antichains, $\n$ is maximal, and the unique maximal ideal of $S$.

Now, let $\p \subset \q$ be a chain of primes in $R$, and let $Q \in \Spec S$ with $Q \cap R = \q$.  Since the Spec map is surjective, there is some $P \in \Spec S$ with $P \cap R = \p$.  By the `going up' property of integral extensions, there is some $Q' \in \Spec S$ such that $P \subseteq Q'$ and $Q' \cap R = \q$.  But then by injectivity of the Spec map, $Q' = Q$.  This shows that the inclusion $R \subseteq S$ satisfies going-down; hence $S$ is a localization of $R$ since $R$ is perinormal.  But since the map $R \rightarrow S$ is a local homomorphism of local rings, the only way $S$ can be a localization of $R$ is if $R=S$.
\end{proof}

Recall that an integral domain $R$ is \emph{weakly normal}\footnote{This is \emph{not} the original definition \cite{AnBom-wn}, but it is equivalent \cite[Remark 1]{Yan-wnext}.} if for any integral overring $S$ of $R$ such that the map $\Spec S \rightarrow \Spec R$ is a bijection and for all $P \in \Spec S$ (where we set $\p := P \cap R$), the corresponding field extensions $R_\p / \p R_\p \rightarrow S_P / PS_P$ is purely inseparable, it follows that $R=S$.

A domain $R$ is \emph{seminormal} if \cite{Sw-semi} whenever $x$ is an element of the fraction field with $x^2, x^3 \in R$, we have $x\in R$.  However, it is equivalent to say that for any integral overring $S$ such that $\Spec S \rightarrow \Spec R$ is a bijection and the corresponding field extensions $R_\p / \p R_\p \rightarrow S_P / P S_P$ are isomorphisms, then $R=S$.  From this, it is clear that every weakly normal domain is seminormal, and that for a domain that contains $\Q$, the converse holds.

Recall that both weak normality and seminormality are local properties in the sense of Proposition~\ref{pr:perilocal}.  Also every normal domain is weakly normal.   For all this and more, cf.\! Vitulli's survey article on weak normality and seminormality \cite{Vit-survey}.

\begin{cor}\label{cor:pnwn}
If $R$ is perinormal, then it is weakly normal $($hence seminormal$)$.
\end{cor}

\begin{proof}
Since both perinormality and weak normality are local properties, we may assume $R$ is local.  Now let $S$ be an integral overring of $R$ where $\Spec S \rightarrow \Spec R$ is a bijection such that for any $P \in \Spec S$, the field extension $R_\p / \p R_\p \rightarrow S_P / P S_P$ is purely inseparable (where $\p =P\cap R$). Then by Proposition~\ref{pr:intbijection}, $R=S$. It follows 
that $R$ is weakly normal.
\end{proof}

We next present two examples to show that the converse to Corollary~\ref{cor:pnwn} is false, even under some additional restrictions.

\begin{example}
Not all weakly normal (resp. seminormal) domains are perinormal, even in dimension 1.  For example, $A=\mathbb R[x,ix]$ is seminormal, even weakly normal, without being perinormal.  Failure of perinormality arises from the fact that $\mathbb C[x]$ is going-down over $A$ (with the same fraction field $\mathbb C(x)$) without being a localization of it.  To see seminormality, merely observe that $A$ consists of those polynomials whose constant term is real, and if $f \in \C[x]$ is such that its square and cube have real constant term, it follows that the constant term of $f$ has its square and cube in $\R$, whence the constant term of $f$ is in $\R$ already.
\end{example}

\begin{example} (Thanks to Karl Schwede for this example.)
Even for finitely generated algebras over algebraically closed fields, weakly normal \Rone\ domains are not necessarily perinormal.  For an example, consider $R=k[x,y,xz,yz, z^2]$ where $k$ is any field of characteristic not equal to 2.  Let $A=k[x,y,z]$, and note that $A$ is the integral closure of $R$.  Hence every prime ideal of $R$ is contracted from $A$.  Let $P \in \Spec A$.  If $P \nsupseteq (x,y)$, then $z\in R_{P \cap R}$, whence $R_{P \cap R} = A_P$ is regular.  Therefore, $R_{P\cap R}$ is normal, weakly normal, and perinormal.  This also shows that $R$ satisfies \Rone.  

Further, Yanagihara \cite[Proposition 1]{Yan-wnext} has shown that an arbitrary pullback of a weakly normal inclusion is also a weakly normal inclusion.  Hence, we may conclude that $R$ is weakly normal, as $R/I$ is a subring of $A/I$, where $I=(x,y,xz,yz)R = (x,y)A$, and the extension $k[z^2] \hookrightarrow k[z]$ is weakly normal, since $\chr k \neq 2$.  (One may similarly show the ring is seminormal even when $\chr k = 2$ by using \cite[4.3]{GrTr-semi} in place of \cite[Proposition 1]{Yan-wnext}, which shows the analogous fact for seminormal inclusions.)

However, the ring $R_{(x,y) \cap R} = k(z^2)[x,y,xz,yz]_{(x,y,xz,yz)}$ is not perinormal.  Its integral closure is $A_{(x,y)} = k(z)[x,y]_{(x,y)}$.  Then the map $R_{(x,y) \cap R} \rightarrow A_{(x,y)}$ induces a bijection on spectra because for all other primes, we have an isomorphism, whereas the localness of the integral closure shows that we also have bijectivity at the maximal ideal.  But the two rings are unequal because $z\notin R_{(x,y) \cap R}$.  Then since $R_{(x,y) \cap R}$ is not perinormal, neither is $R$.
\end{example}

\begin{lemma}\label{lem:R1over}
Let $R$ be an integral domain, $S$ an overring of $R$, and $\p \in \Spec S$ such that $V := R_{\p \cap R}$ is a valuation domain of dimension 1.  Then $R_{\p \cap R} = S_\p$ as subrings of the fraction field $K$ of $R$, and $\hgt \p = 1$.
\end{lemma}

\begin{proof}
We have $V=R_{\p \cap R} \subseteq S_\p \subseteq K$.  But $S_\p \neq K$, since $\p\neq 0$.  On the other hand, $V$ is a valuation domain, so every overring is a localization at a prime ideal.  Since $V$ has only two primes, the only possibilities are $V$ and $K$.  Since $S_\p \neq K$, it follows that $S_\p = V$.  Finally, $\hgt \p = \dim S_\p = \dim V = 1$.
\end{proof}

\begin{defn}
For a commutative ring $R$, $\Spec^1(R)$ denotes the set of all \emph{height one} primes of $R$.
\end{defn}

\begin{prop}\label{pr:Spec1}
Let $R$ be an \Ronei\ domain and let $S$ be an overring such that the extension $R \subseteq S$ satisfies going-down.  Then $S$ satisfies \Ronei, and the map $\Spec S \rightarrow \Spec R$ induces an injective map $\Spec^1(S) \rightarrow \Spec^1(R)$ whose image consists of those height one primes $\p$ of $R$ such that $\p S \neq S$.
\end{prop}

\begin{proof}
First we need to show that given a height one prime $Q$ of $S$, $\q:=Q \cap R$ is a height one prime of $R$.  We have $\q \neq 0$ because $R \subseteq S$ is an essential extension of $R$-modules; hence $\hgt \q \geq 1$.  On the other hand, suppose there is some $\p \in \Spec R$ with $0 \subsetneq \p \subsetneq \q$.  Then by going-down, there is some $P \in \Spec S$ with $P \cap R = \p$.  But then $P \neq 0$ (again by essentiality of the extension), whence $0 \subsetneq P \subsetneq Q$ is a chain of primes in $S$, so that $\hgt Q \geq 2$, a contradiction.  Then by Lemma~\ref{lem:R1over}, $S$ satisfies \Rone.

Next, let $\p \in \Spec^1(R)$.  If $\p S = S$, then no prime of $S$ can lie over $\p$.  On the other hand, if $\p S \neq S$, then there is some maximal ideal $Q$ of $S$ with $\p S \subseteq Q$.  Then the going-down property implies that there is some $P \in \Spec S$ with $P \cap R = \p$.  Moreover, Lemma~\ref{lem:R1over} along with the \Rone-ness of $R$ implies that $S_P = R_\p$ and $\hgt P = 1$.  Finally, if there is some other prime ideal $P'$ of $S$ with $P'\cap S= \p$, then we have $S_P = R_\p = S_{P'}$.  But different prime ideals of a ring always give rise to different localizations, so $P=P'$, finishing the proof that the map of $\Spec^1$'s is injective.
\end{proof}

Consider the following properties that an integral domain $R$ may have:
\begin{enumerate}
\item $R = \bigcap_{\p \in \Spec^1(R)} R_\p$.
\item For any nonzero element $r\in R$, the set $\{\p \in \Spec^1(R) \mid r\in \p\}$ is finite.
\item $R_\p$ is a DVR for all $\p \in \Spec^1(R)$.
\end{enumerate}
One says $R$ is a \emph{Krull domain} (resp. \emph{generalized Krull domain}) if it satisfies properties (1--3) (resp. properties (1), (2), and \Rone).

Recall the \emph{Mori-Nagata theorem} (cf. \cite[Theorem 4.3]{Fos-Cl}), which says that the integral closure of any Noetherian domain is Krull (though not necessarily Noetherian); hence every Noetherian normal domain is Krull.  

\begin{thm}\label{thm:Krperi}
If $R$ is a generalized Krull domain $($e.g. Noetherian normal$)$, then $R$ is perinormal.
\end{thm}

\begin{proof}
Let $(S, \m)$ be a local overring of $R$ such that the inclusion $R \subseteq S$ satisfies going-down.  Let $Q = \m \cap R$; $R_Q$ is then also a generalized Krull domain \cite[Corollary 43.6]{Gil-MIT}.  Note that the going-down condition implies that the map $\Spec S \rightarrow \Spec R_Q$ is surjective.  Hence by Proposition~\ref{pr:Spec1}, we get a \emph{bijective} map $\Spec^1(S) \rightarrow \Spec^1(R_Q)$, and for each $P \in \Spec^1(S)$ and corresponding $\p=P\cap R \in \Spec^1(R_Q)$, we have $(R_Q)_\p = S_P$ by Lemma~\ref{lem:R1over}.  Therefore \[
R_Q \subseteq S \subseteq \bigcap_{P \in \Spec^1(S)} S_P = \bigcap_{\p \in \Spec^1(R_Q)} (R_Q)_\p = R_Q.
\]
That is, $S=R_Q$, so $R$ is perinormal.
\end{proof}

\section{Local characterizations of perinormality}\label{sec:locNoeth}
In this section, after a preliminary exploration of how \Rone\ domains interact with overrings and the special relationship that occurs between two rings that share a nonzero ideal, we provide two surprising characterizations of perinormal domains within a large class of integral domains.

\begin{lemma}\label{lem:iclocal}
Let $R$ be an \Ronei\ integral domain whose integral closure $R'$ is a generalized Krull domain and such that for all $P\in \Spec^1 R'$, $P\cap R \in \Spec^1 R$.  If there is a maximal ideal of $R$ that contains all the height one primes of $R$, then $R$ is local.
\end{lemma}

\begin{proof}
Let $\m$ be a maximal ideal of $R$, and suppose that $\m$ contains all height one primes of $R$.  By Lemma ~\ref{lem:R1over}, since $R$ is an \Ronei\ domain, if $\p \in \Spec^1 R$ and $P\in \Spec^1 R'$ lies over $\p$, then $R_\p=R'_P$.  As height one primes of $R'$ contract to height one primes of $R$,  we have
 \[R_\m \subseteq \bigcap_{\p \in \Spec^1(R_\m)} (R_\m)_\p = \bigcap_{P \subseteq \m, \hgt P =1} R_P = \bigcap_{P \in \Spec^1R} R_P = R',
\]
 where the last equality follows since $R'$ is generalized Krull.  We have shown that $R_\m$ is integral over $R$, which can only happen if $R=R_\m$.
\end{proof}

\begin{lemma}\label{lem:lochomo}
Let $(R,\m)$ be an \Ronei\ local domain whose integral closure $R'$ is a generalized Krull domain and such that for all $P\in \Spec^1 R'$, $P\cap R \in \Spec^1 R$.  Let $S$ be an integral overring of $R$ that satisfies going-down over $R$.  Then $S$ is local.
\end{lemma}

\begin{proof}
By Proposition~\ref{pr:Spec1}, $S$ satisfies \Rone\ and the map $\Spec S \rightarrow \Spec R$ induces a bijection $\Spec^1(S) \arrow{\sim} \Spec^1(R)$.  Now let $\n$ be a maximal ideal of $S$ that contains $\m S$; it exists because $S$ is integral over $R$.  Then the extension $R \subseteq S_\n$ is going-down and $\m S_\n \neq S_\n$, so Proposition~\ref{pr:Spec1} applies again to produce a bijection $\Spec^1(S_\n) \arrow{\sim} \Spec^1(R)$.  This bijection composes with the inverse of the previous bijection to give a bijection of $\Spec^1(S_\n)$ with $\Spec^1(S)$.  Hence, for all height one primes $\p$ of $S$, we have $\p S_\n \neq S_\n$ -- that is, $\p \subseteq \n$.  Moreover, $R$ and $S$ have the same integral closure, which is generalized Krull by hypothesis.  As height one primes of $R'$ contract to height one primes of $R$, one can show using the properties of integrality that the same holds for all intermediate rings.  Thus by Lemma~\ref{lem:iclocal}, $S$ must be local.
\end{proof}

Next we give conditions on a domain $A$  that ensure that height one primes of $A'$ contract to height one primes of $A$.  We first need to recall some definitions. A ring $A$ is called {\it catenary} if given a pair $\p_1 \subset\p_2 $ of prime ideals of $ A$ such that there exists a saturated chain of prime ideals between the two, then all such saturated chains  have the same length.  We say that $A$ is {\it universally catenary} if it is Noetherian and every finitely generated $A$-algebra is catenary.  It is clear that being catenary and hence being universally catenary is closed under localization.

 Let $A \subseteq B$ be integral domains.  Then tr.deg$_AB$ denotes the transcendence degree of the fraction field of $B$ over that of $A$.  Recall that the ring extension is said to satisfy  the dimension  (or altitude) formula if the following equality holds for all $P\in \Spec B$:
$$ \mbox{ht} P + \mbox{tr.deg}_{A/\p} B/P = \mbox{ht } \p + \mbox{tr.deg}_A B$$
 where $\p = P\cap A$ (see for example \cite[p. 119]{Mats}).  We note that if in addition $B$ is integral over $A$, then tr.deg$_A B =0=$ tr.deg$_{A/\p} B/P$ in which case the height of a prime of $B$ is invariant under contraction to $A$.

\begin{lemma} \label{lem:save}
Let $R$ be a  universally catenary integral domain with integral closure $R'$. Then every height one prime ideal of $R'$ contracts to a height one prime ideal of $R$.
\end{lemma}
\begin{proof}    Since $R$ is Noetherian, by \cite[Corollary 2.5]{Rat-3int} it will suffice to show  that if $f\in R[x]'$, then the height of a prime ideal of $R[x,f]$ is invariant under contraction to $R[x]$.  Since $R[x]$ is also universally catenary and $R[x,f]$ is module finite over $R[x]$ (in particular algbera finite), it follows that the extension $R[x] \subseteq R[x,f]$ satisfies the dimension formula (see for example \cite[Theorem 15.6]{Mats}).   Since it is also an integral extension, we have that the height is invariant under contraction as desired.
\end{proof}

\begin{rmk*} The universal catenarity assumption is not particularly restrictive, as almost every Noetherian ring that arises in algebraic geometry, number theory, and everyday commutative algebra is universally catenary.  Indeed, the class of universally catenary rings is closed under localization and finitely generated algebra extensions.  Moreover, it includes 
 Cohen-Macaulay rings (including Dedekind domains, fields, and regular local rings; see \cite[Theorem 17.9]{Mats}) and 
complete Noetherian local rings (e.g. power series rings in finitely many variables over a field or over the $p$-adics; see \cite[Theorem 29.4]{Mats}).
\end{rmk*}

The following three results are well-known to experts, and some of their statements appear (without proofs) in \cite{Fo-topdef}.  However, we include them here for completeness and to make the paper self-contained.

\begin{lemma}\label{lem:locms}
Let $R \subseteq T$ be an inclusion of commutative rings, and let $I$ be an ideal that is common to $R$ and $T$. $($That is, $I$ is an ideal of $R$ and $IT=I$.$)$  Let $W$ be a multiplicatively closed subset of $T$, set $V:=W \cap R$, and suppose that $I \cap W\neq \emptyset$.  Then the natural map $R_V \rightarrow T_W$ is an isomorphism.
\end{lemma}

\begin{proof}
Let $z\in I \cap W$.  	To see injectivity, let $\frac{r}{v} \in R_V$ (with $r\in R$, $v\in V$) such that $\frac{r}{v}=0$ in $T_W$.  Then for some $w\in W$, we have $wr=0$.  Moreover, $zw \in I \cap W \subseteq R \cap W = V$ and $(zw)r=0$, whence $\frac{r}{v}=0$ in $R_V$.

To see surjectivity, let $\frac{t}{w} \in T_W$ (with $t\in T$, $w\in W$).  Then $zt \in IT \subseteq R$ and $zw \in I \cap W \subseteq R \cap W = V$, so that $\frac{t}{w} = \frac{zt}{zw} \in R_V$.
\end{proof}

\begin{cor}\label{cor:commonideal}
Let $R \subseteq T$ be an inclusion of commutative rings, and let $I$ be an ideal common to $R$ and $T$.  Let $z\in I$, and let $P \in \Spec T$ with $I \nsubseteq P$.  Then the natural maps $R_z \rightarrow T_z$ and $R_{P \cap R} \rightarrow T_P$ are isomorphisms.
\end{cor}

\begin{proof}
In the first case, apply Lemma~\ref{lem:locms} with $V=W=\{z^n \mid n\in \N\}$.  In the second case,  apply the same lemma with $W=T \setminus P$.
\end{proof}

\begin{cor}\label{cor:samefrac}
Let $R \subseteq T$ be integral domains that share a common nonzero ideal $I$.  Then the induced map of fraction fields is an isomorphism.
\end{cor}

\begin{proof}
Apply Lemma~\ref{lem:locms} with $W = T \setminus \{0\}$.
\end{proof}

\begin{thm}\label{thm:equiv}
Let $R$ be a universally catenary integral domain with fraction field $K$. The following are equivalent.  \begin{enumerate}[label=\emph{(\alph*)}]
\item $R$ is perinormal.
\item For each $\p \in \Spec R$, $R_\p$ is the only ring $S$ between $R_\p$ and $K$ such that the induced map $\Spec S \rightarrow \Spec R_\p$ is an order-reflecting bijection.
\item $R$ satisfies \Ronei, and for each $\p \in \Spec R$, $R_\p$ is the only ring $S$ between $R_\p$ and its integral closure such that the induced map $\Spec S \rightarrow \Spec R_\p$ is an order-reflecting bijection.
\end{enumerate}
\end{thm}

\begin{proof} We note that we only need the universal catenarity condition for the implication (c) $\implies$ (a).

(a) $\implies$ (b): Since perinormality localizes, we may assume that $(R,\p)$ is local.  Now let $S$ be a ring between $R$ and $K$ such that $\Spec S \rightarrow \Spec R$ is an order-reflecting bijection.  Thus $S$ satisfies going-down over $R$.   Since $S$ is local, the perinormality assumption on $R$ implies that $S$ is a localization of $R$.  As the Spec map is onto, we conclude that $R=S$.

(b) $\implies$ (c): To see that $R$ satisfies \Rone, let $\p$ be a height one prime of $R$.  Let $V$ be a valuation ring centered on $\p$.  Then all nonzero prime ideals of $V$ contract to $\p$, and their intersection $\q$ is also a prime ideal of $V$.  Since $\q$ contains no prime ideals other than itself and $(0)$, we have $\hgt \q =1$.  Now, the map $R_\p \rightarrow V_\q$ induces a bijection on Spec, so $R_\p = V_\q$, a valuation domain.  On the other hand, the second condition in (c) follows trivially from (b).

(c) $\implies$ (a): Let $(S,\n)$ be a going-down local overring of $R$.
  Let $\p = \n \cap R$.  Note that $R_\p$ satisfies \Rone, so that by Proposition~\ref{pr:Spec1}, the map $\Spec S \rightarrow \Spec R_\p$ induces a bijection $\Spec^1(S) \arrow{\sim} \Spec^1(R_\p)$ where by Lemma~\ref{lem:R1over}, the corresponding localizations of $S$ and $R_\p$ coincide.  Hence $S$ is \Rone.  Since  $R_\p$ is also universally catenary  it follows  by integrality and Lemma~\ref{lem:save} that the Spec map $\Spec (R_\p)' \rightarrow \Spec R_\p$ induces a bijection $\Spec^1 (R_\p)' \arrow{\sim} \Spec^1 R_\p$, where again the corresponding localizations of $R_\p$ and $(R_\p)'$ coincide. Thus, \[
R_\p \subseteq S \subseteq \bigcap_{Q \in \Spec^1(S)} S_Q = \bigcap_{P \in \Spec^1(R_\p)} (R_\p)_P = (R_\p)',
\]
where the last equality follows since $(R_\p)'$ is a Krull domain.   Hence $S$ is  integral over $R_\p$.

Next, we claim that the map $\Spec S \rightarrow \Spec R_\p$ is injective.  To see this, let $Q$ be a prime ideal of $R_\p$, and let $W := R_\p \setminus Q$.  Then the inclusion $(R_\p)_Q \subseteq S_W$ is integral, it satisfies going-down, and $Q S_W \neq S_W$.  Moreover, the integral closure of $R_\p$ is $R'_{R \setminus \p}$, a Krull domain.  Thus by Lemmas~\ref{lem:lochomo} and ~\ref{lem:save}, $S_W$ is local.  But this means that only one prime of $S$ lies over $Q$, whence the map $\Spec S \rightarrow \Spec R_\p$ is injective.

However, the map is also surjective, since $S$ is integral over $R_\p$.  Therefore the map is bijective.

Finally, we must show that the map is order-reflecting -- that is, if $\q_1 \subseteq \q_2$ in $R_\p$, then the corresponding primes in $S$ are also so ordered.  So let $Q_j \in \Spec S$ with $Q_j \cap R_\p = \q_j$, $j=1,2$.  By going-down, there is some $P \in \Spec S$ with $P \subseteq Q_2$ and $P \cap R_\p = \q_1 = Q_1 \cap R_\p$.  But then by the injectivity of the Spec map, $P=Q_1$, whence $Q_1 \subseteq Q_2$.  Hence, condition (c) applies and $R_\p = S$, whence $R$ is perinormal.
\end{proof}

Recall \cite{FeOl-min} that a ring extension $A \subseteq B$ is called \emph{minimal} if there are no rings properly between $A$ and $B$.

\begin{cor} \label{cor:min}
Let $(R,\m)$ be a universally catenary local domain.
 Assume that $\dim R \geq 2$ and that the map $R \rightarrow S$ is a minimal ring extension, where $S$ is the integral closure of $R$.  Then $R$ is perinormal if and only if $S$ is not local.
\end{cor}

\begin{proof}
By \cite[Theorem 2.2]{FeOl-min}, $\m$ is also an ideal of $S$.  Now let $\p \in \Spec R$ with $\p \neq \m$.  Let $P, P' \in \Spec S$ with $P \cap R = P'\cap R = \p$.  Then by Corollary~\ref{cor:commonideal}, $S_P = R_\p = S_{P'}$, whence $P=P'$.  Also, by integrality any maximal ideal of $S$ must contract to $\m$.  Hence, there is a bijection between the nonmaximal primes of $R$ and those of $S$.

Suppose $S$ is local. The only possibility of non-bijection of Spec happens at the maximal ideals, but it is clear that the unique maximal ideal of $S$ contracts to $\m$.  Thus, $R \rightarrow S$ induces a bijection on Spec even though $R \neq S$.  Then by the implication (a) $\implies$ (b) of Theorem~\ref{thm:equiv}, $R$ cannot be perinormal.

On the other hand, if $S$ is not local, then by minimality of the extension, there is no local integral overring of $R=R_\m$ other than $R$ itself.  Also, for any $\p \in \Spec R \setminus \{\m\}$, $R_\p$ is integrally closed (because it equals $S_P$, where $P$ contracts to $\p$), so again there is no local integral overring.  The same observation shows that $R$ satisfies \Rone, since none of the height one primes of $R$ are maximal.  Then by the implication (c) $\implies$ (a) of Theorem~\ref{thm:equiv}, $R$ is perinormal.
\end{proof}

We close this section by presenting an example that shows that Theorem ~\ref{thm:equiv} and Corollary ~\ref{cor:min} are in some sense best possible.

\begin{example}
The fact that  the last two results are false for arbitrary Noetherian rings can be demonstrated by 
\cite[Appendix, Example 2]{NagLR} with $m=0$.  This example consists of a Noetherian normal ring $S$ with exactly two maximal ideals $\m_1$ and $\m_2$ where ht $\m_1=1$ and ht $\m_2=2$ and a field $k\subset S$ such that the canonical map $k\to S/\m_i$ is an isomorphism for $i=1,2$.  If $R = k+J$, where $J=\m_1\cap \m_2$, then Nagata shows that $S$ is the integral closure of $R$. 

We next claim that the set $\Spec^1 R$ is in natural bijection with the set $X$ of height one primes contained in $\m_2$, and that the corresponding localizations are equal.  To see this, let $\p \in \Spec^1 R$.  Then by integrality, there is some $P \in \Spec^1 S$ with $P \cap R = \p$.    But all height one primes of $S$ are in $\m_2$ except $\m_1$, and $\m_1 \cap R=J \supsetneq \p$.  Thus, $P \subset \m_2$.  Hence, contraction gives a surjective map $X \twoheadrightarrow \Spec^1R$.  Finally, if $P, P' \in X$ with $P \cap R = P'\cap R=\p$, then $S_P = R_\p= S_{P'}$ (by Corollary~\ref{cor:commonideal}), whence $P=P'$.

Hence $R$ satisfies \Rone\ and the ring $S_{\m_2}$ satisfies going-down over $R$.   However, the ring $S_{\m_2}$ cannot be a localization of $R$ as the maximal ideal of the former contracts to the maximal ideal of $R$. Therefore $R$ is not perinormal.  On the other hand, $S$ is a minimal ring extension of $R$ by \cite[Theorem 3.3(b)]{PP-amin}.  Hence there are no local rings strictly between $R$ and $R'=S$ and so condition (c) of Theorem ~\ref{thm:equiv} is satisfied.   Thus both Theorem ~\ref{thm:equiv} and Corollary ~\ref{cor:min} are false without some assumption on $R$.
\end{example}
\section{Gluing points of generalized Krull domains in high dimension}\label{sec:gluing}

In this section, we exhibit a method for constructing perinormal domains out of pre-existing generalized Krull domains, such that the new domains enjoy an arbitrary degree of branching-like behavior.  We explain how to interpret this construction either in the algebraic context of pullbacks or the geometric context of gluing points.

We begin with the following result, which may be known, but we include a proof for the convenience of the reader.

\begin{lemma}\label{lem:gdtower}
Let $R \subseteq S \subseteq T$ be ring extensions.  Let $X := \{P \in \Spec S \mid P \cap R\text{ is a maximal ideal}\}$.  Suppose that the induced map $(\Spec S \setminus X) \rightarrow (\Spec R \setminus \Max R)$ is injective and the inclusion $R \subseteq S$ satisfies INC.  If $R \subseteq T$ satisfies going-down, so does $S \subseteq T$.
\end{lemma}

\begin{proof}
Let $P_1 \subset P_2$ be a chain of two prime ideals of $S$ such that there exists $Q_2 \in \Spec T$ with $Q_2 \cap S = P_2$.  Then setting $\p_j := P_j \cap R$, $j=1,2$, we have $\p_1 \neq \p_2$ (by INC) and $Q_2 \cap R = \p_2$.  Then by the going-down hypothesis on the extension $R \subseteq T$, there is a prime ideal $Q_1$ of $T$ with $Q_1 \subseteq Q_2$ and $Q_1 \cap R = \p_1$.  But then we have $P_1 \cap R = \p_1 = Q_1 \cap R = (Q_1 \cap S) \cap R$, so by injectivity of the map in question (since $\p_1$ is a non-maximal ideal of $R$), we have $P_1 = Q_1 \cap S$, completing the proof.
\end{proof}

\begin{thm}\label{thm:gluing}
Let $S$ be a semilocal generalized Krull domain and let $\m_1, \dotsc, \m_n$ be its maximal ideals.   Assume that $n\geq 2$, and  $\hgt \m_j \geq 2$ for all $1\leq j\leq n$.  Further suppose that the fields $S/\m_i$ are all isomorphic to the same field $k$.  For each $i=1,2,\ldots,n$ fix an isomorphism $\alpha_i:k \to S/\m_i$.
 Let $R$ be the pullback in the diagram
 \[
\xymatrix{R \ar[r]^f \ar[d]^g
& S \ar[d]^p \\
k \ar[r]^h & S/J}
\]
where $J := \m_1 \cap \cdots \cap \m_n = \prod_{j=1}^n \m_j$, $p$ is the canonical projection, and $h$ is the composition of the maps $k\to \prod_{i=1}^n S/\m_i$ $($given by $\lambda \mapsto (\alpha_1(\lambda), \ldots,\alpha_n(\lambda))$ and the isomorphism between $\prod_{i=1}^n S/\m_i$ and $S/J$ $($given by the Chinese Remainder Theorem$)$.
 Then $R$ is local and perinormal.  Also, $R$ is globally perinormal if $S$ is.  But $R$ is not integrally closed, because its integral closure is $S$.
\end{thm}

\begin{proof} We first note that it follows from the properties of a pullback that as $h$ is an injection (resp. $p$ is a surjection),  $f$ is an injection (resp. $g$ is a surjection).  Thus we can view $R$ as a subring of $S$ where $J=$ ker $g$  is a common nonzero ideal of both rings.  Then it follows from Corollary~\ref{cor:samefrac} that $R$ and $S$ have the same field of fractions.

Next we show  that $S$ is integral over $R$ (and hence equals the integral closure of $R$).  To see this, let $s\in S$.  Since $J$ is a common ideal of $R$ and $S$, we have \[
k \cong R/J \hookrightarrow S/J = S/(\m_1 \cap \ldots \cap \m_n) \cong \prod_{j=1}^n (S/\m_j) \cong k \times \cdots \times k,
\]
where the composite map is just the diagonal embedding.  Now $k \times \cdots \times k$ is integral over $k$, which means that $S/J$ is integral over $R/J$.  In particular, there is some monic $g \in (R/J)[X]$ such that $g(\bar s) = 0$.  But then $g$ lifts to a monic polynomial $G \in R[X]$ such that $G(s) \in J$.  Say $G(s) = j \in J$.  Then $H(X) := G(X) - j$ is a monic polynomial over $R$ such that $H(s)=0$.
It follows that the integral closure of $R$ is $S$.

Now we claim that $R$ is local.  This will follow if we can show that $J$ is the Jacobson radical of $R$, since we already have that $J$ is a maximal ideal of $R$.  To this end, it suffices to show that for each $j\in J$, $1-j$ is a unit of $R$.   If not, then $1-j \in \p$ for some prime ideal of $R$, so that $1-j \in \p S$.  But since $1-j$ is a unit of $S$ (since $J$ is the Jacobson radical of $S$), it follows that $\p S = S$, which contradicts the lying over property of the integral extension $R \subseteq S$.  This contradiction proves the claim.

Before showing that $R$ is perinormal, we collect some observations about the relationship between $\Spec R$ and $\Spec S$.
Let $P\in\Spec S$ and  $\p=P\cap R$.   By integrality $P$ is a nonmaximal ideal of $S$ if and only if $\p$ is a nonmaximal ideal of $R$.  Furthermore in this case by Corollary~\ref{cor:commonideal}, we have $R_\p = S_P$,  whence ht $P = $ ht $\p$.   Since we are assuming that no maximal ideal of $S$ has height 1, each height one prime of $S$ must contract to a height one prime of $R$.  Moreover  by integrality each height one prime of $R$ is lain over by a prime of $S$.   Thus the Spec map induces a bijection $ \Spec^1(S)\to \Spec^1(R)$ where the corresponding localizations coincide.  In particular $R$ satisfies \Rone.

Now let $T$ be an overring of $R$ such that $R \subseteq T$ satisfies going-down.

\noindent\textbf{Case 1:} Suppose $JT=T$.  Then $S \subseteq T$.  To see this, let $\n$ be a maximal ideal of $T$.  Since $J \nsubseteq \n$, we have $\n \cap R = \p \subsetneq J$.  Then there is some nonmaximal $P \in \Spec S$ with $P \cap R = \p$ (since $S$ is integral over $R$), whence  we have $R_\p = S_P$.  Hence, $S \subseteq S_P = R_\p \subseteq T_\n$.  Since $\n$ was an arbitrary maximal ideal of $T$, it follows that $S \subseteq \bigcap_{\n \in \Max T} T_\n = T$.

Next, since the map $\Spec S \rightarrow \Spec R$ is injective on non-maximal ideals, $S$ is integral over $R$, and $R \subseteq T$ satisfies going-down, it follows from Lemma~\ref{lem:gdtower} that the extension $S \subseteq T$ satisfies going-down.
Thus, if $T$ is local or $S$ is globally perinormal, we have that $T = S_W$ for some multiplicative subset $W$ of $S$.  On the other hand, for any maximal ideal $\m_i$ of $S$, we have $S_W = T=JT \subseteq \m_i T = \m_i S_W$, so $\m_i \cap W \neq \emptyset$.  Let $z_i \in \m_i \cap W$, and let $z := \prod_{i=1}^n z_i$.  Note that $z \in J$ and that $z$ is a unit in $T$.  Thus by Corollary~\ref{cor:commonideal}, $T=S_W = (S_z)_V = (R_z)_V = R_{V'}$ for appropriate multiplicative sets $V$ and $V'$, so that $T$ is a localization of $R$.

\noindent\textbf{Case 2:} On the other hand if $JT \neq T$, then by Proposition~\ref{pr:Spec1}, the map $\Spec T \rightarrow \Spec R$ induces a bijection $\Spec^1(T) \arrow{\sim} \Spec^1(R)$, and by Lemma ~\ref{lem:R1over} the corresponding localizations are equal.  Since we have a similar bijection between  $\Spec^1(S)$ and $\Spec^1(R)$, we get
 \[R \subseteq T \subseteq \bigcap_{P \in \Spec^1(T)} T_P = \bigcap_{Q \in \Spec^1(S)} S_Q = S.
\]
where the last equality follows from $S$ being a generalized Krull domain. It follows that $T$ is integral over $R$ and that $J$ is a common ideal to $R$, $S$, and $T$, so we have \[
k \cong R/J \subseteq T/J \subseteq S/J \cong k \times \cdots \times k.
\]
Thus, $T/J$ must be isomorphic to a product of some finite number of copies of $k$.  But by Lemma~\ref{lem:lochomo}, $T$ is local.  Hence, $T/J$ is local as well.  Therefore, $T/J \cong k$, whence $T=R$.
\end{proof}

\begin{example}\label{ex:gluepoints}
For a geometrically relevant example of the above, let $B = k[X,Y]$, let $p_j = (x_j, y_j) \in k^2$ be distinct ordered pairs (points of $k^2$) for $1\leq j \leq t$, let $\n_j := (X-x_j, Y-y_j)$ (the maximal ideal corresponding to $p_j$), $J :=\bigcap_{j=1}^t \n_j$, and $A :=k+\bigcap_{j=1}^t \n_j$.  Note that $J$ is a maximal ideal of $A$.  Then by the above theorem, the ring $A_J$ is perinormal (even globally perinormal!), but it isn't normal unless $t=1$ (since there are $t$ maximal ideals lying over $JA_J$ in the integral closure of $A_J$.)

By \cite[Th\'eor\`eme 5.1]{Fe-pinch}, $\Spec A$ can be seen, quite precisely, as the algebro-geometric result of gluing together the points $p_1, \ldots, p_t$ of $\A^2_k$ together, and $\Spec A_J$ is the (geometric) localization at the resulting singular point.
\end{example}

\section{Global perinormality}\label{sec:global}
Next, we explore the related but quite distinct concept of \emph{global} perinormality.  In particular, for Krull domains, there is a strong and surprising relationship to the divisor class group.  We illustrate with examples from algebraic number theory.

\begin{prop}
Let $R$ be a globally perinormal domain, and let $W$ be a multiplicative subset of $R$.  Then $R_W$ is globally perinormal as well.
\end{prop}

\begin{proof}
Let $S$ be an overring of $R_W$ that satisfies going-down.  Let $Q \in \Spec S$.  Then by Lemma~\ref{lem:surj}, the map $\Spec S_Q \rightarrow \Spec (R_W)_{Q \cap R_W}$ is surjective.  But $(R_W)_{Q \cap R_W} = R_{Q \cap R}$ canonically, so that the map $\Spec S_Q \rightarrow \Spec R_{Q\cap R}$ is surjective.  Since $Q\in \Spec S$ was arbitrarily chosen, Lemma~\ref{lem:surj} applies again to show that the map $R \rightarrow S$ satisfies going-down, whence since $R$ is globally perinormal, $S$ must be a localization of $R$.  That is, $S = R_V$ for some multiplicative subset $V$ of $R$.  But since $R_W \subseteq S$, we have $W \subseteq V$, so that $V' := VR_W$ is a multiplicative subset of $R_W$, and $S=R_V = (R_W)_{V'}$, finishing the proof that $R_W$ is globally perinormal.
\end{proof}

We next give a result analogous to Proposition~\ref{pr:periflat}.

\begin{prop}\label{pr:gperiflat}
Let $R$ be a perinormal domain.  Then $R$ is globally perinormal if and only if every flat overring of $R$ is a localization of $R$.
\end{prop}

\begin{proof}
Suppose $R$ is globally perinormal.  Let $S$ be a flat overring of $R$.  Then $S$ satisfies going-down over $R$ (by flatness), so by global perinormality, $S = R_W$ for some multiplicative set $W \subseteq R$.

Conversely, suppose every flat overring of $R$ is a localization of $R$.  Let $S$ be an overring that satisfies going-down over $R$.  By perinormality and Proposition~\ref{pr:periflat}, $S$ is flat over $R$, whence by assumption $S$ is a localization of $R$.  Hence, $R$ is globally perinormal.
\end{proof}

\begin{prop}\label{pr:delta}
Let $R$ be a generalized Krull domain and let $S$ be a going-down overring of $R$.  Then \[
S = \bigcap_{\p \in \Delta} R_\p =: R_\Delta,
\]
where $\Delta := \{\p \in  \Spec^1(R) \mid \p S \neq S\}$.
\end{prop}

\begin{proof}
For any maximal ideal $\m$ of $S$, $S_\m$ is local overring of $R$ such that $R \subseteq S_\m$ satisfies going-down.  Hence by Theorem~\ref{thm:Krperi}, $S_\m$ is a localization of $R$ -- i.e., $S_\m = R_{\m \cap R}$.  Now, for every $\p \in \Delta$, there is some such $\m \in \Max S$ with $\p S \subseteq \m$, whereas when $\p \in \Spec^1(R) \setminus \Delta$, there is no such $\m$.  Also, every such $R_{\m \cap R}$ is generalized Krull, by \cite[Corollary 43.6]{Gil-MIT}.  
Thus: \begin{align*}
S &= \bigcap_{\m \in \Max S} S_\m = \bigcap_{\m \in \Max S} R_{\m \cap R} = \bigcap_{\m \in \Max S} \left(\bigcap_{P \in \Spec^1(R_{\m \cap R})} (R_{\m\cap R})_P \right)\\
&= \bigcap_{\m \in \Max S} \left(\bigcap_{\p \in \Spec^1(R),\ \p S \subseteq \m} R_\p \right) = \bigcap_{\p \in \Delta} R_\p = R_\Delta.
\end{align*}
\end{proof}


The next theorem involves the divisor class group; hence we restrict our attention to Krull domains (rather than generalized Krull domains), which is where the theory of the divisor class group is most well developed.

\begin{thm}\label{thm:class}
Let $R$ be a Krull domain.  \begin{enumerate}[label=\emph{(\arabic*)}]
\item\label{it:classglobal} If the divisor class group $\Cl(R)$ of $R$ is torsion, then $R$ is globally perinormal.
\item\label{it:classglobal1} The converse holds when $\dim R=1$.
\end{enumerate}
\end{thm}

\begin{proof}
To prove part (1), it is enough to show that ever flat overring of $R$ is a localization of $R$, due to Proposition~\ref{pr:gperiflat} and Theorem~\ref{thm:Krperi}.  So let $S$ be a flat overring of $R$.  Then by Theorem~\ref{thm:Rich}, $S_\m = R_{\m \cap R}$ for all maximal ideals $\m$ of $S$.  In particular, $S$ is an intersection of localizations of $R$ at prime ideals of $R$.  But recall that Heinzer and Roitman prove \cite[Corollary 2.9]{HeiRo-center} that for a Krull domain with torsion divisor class group, any intersection of localizations of $R$ at prime ideals is in fact a localization of $R$.  Thus, $S$ is a localization of $R$, whence $R$ is globally perinormal.

As for part (2), the following statement was proved independently in \cite[Theorem 2]{Dav-over2}, \cite[Corollary (1)]{Gol-Dedekind}, and \cite[Corollary 2.6]{GilOhm-quot}: \begin{quotation}
 Let $R$ be a Dedekind domain.  Then the class group $\Cl(R)$ of $R$ is torsion if and only if every overring of $R$ is a localization of $R$.
\end{quotation}
But any 1-dimensional Krull domain is a Dedekind domain \cite[Theorem 12.5]{Mats}. Hence, (2) follows.
\end{proof}

\begin{example}
The ring of integers $\cO_K$ of any finite algebraic extension of $K$ of $\Q$ is globally perinormal.  This is because $\cO_K$ is a Dedekind domain (hence Krull) with finite (hence torsion) class group (cf. \cite[Theorem 31]{FrTay-ANT}).  The result then follows from Theorem~\ref{thm:class}.
\end{example}

\begin{example}
If $R_\m$ is globally perinormal for all maximal ideals $\m$, it does not follow that $R$ is globally perinormal, even when $R$ is a Dedekind domain finitely generated over a field.  To see this, let $E$ be any elliptic curve, with Weierstra\ss\ equation $f=0$, considered as an affine curve in $\A^2_\C$.  Then as a group, $E=E(\C)$ is analytically isomorphic (as an algebraic group) to $\C/\Lambda$ for some lattice $\Lambda$ \cite[Corollary VI.5.1.1]{Sil-arithbook}, which in turn is abstractly isomorphic (as a group) to $\R/\Z \times \R/\Z$.  The latter has uncountably many non-torsion elements (namely, whenever either of the two coordinates is irrational).  On the other hand, $E(\C)$ is isomorphic to a particular subgroup (the so-called \emph{degree 0 part}) of the divisor class group of the Dedekind domain $R=\C[X,Y]/(f)$ \cite[Proposition III.3.4]{Sil-arithbook}, as the latter is the affine coordinate ring of $E(\C)$.  Thus, $\Cl(R)$ contains (uncountably many) non-torsion elements, so by Theorem~\ref{thm:class}(2), $R$ is not globally perinormal.  But $R_\m$ is a DVR for any $\m \in \Max R$ (since $R$ is a Dedekind domain), so $R_\m$ is globally perinormal.
\end{example}

On page~\pageref{fig:1}, we have constructed a chart tracking many of the dependencies we have discussed so far.  Note that none of the arrows are reversible, and that a crossed-out arrow indicates a specific non-implication.
\begin{figure}\label{fig:1}
\[
\begin{tikzcd}[arrows=Rightarrow]
& & \tikst{Dedekind} \ar{dll} \ar[crossout]{ddr}\\
\tikst{Noetherian\\ normal}\ar{dr} &\tikst{$\cO_K$, where $K$ is an\\algebraic number field}\ar{ur} \ar{r} & \tikst{Krull, with\\ torsion $\Cl(R)$} \ar{dl} \ar{dr} &  \\
& \tikst{Krull}\ar{d}
& & \tikst{globally\\ perinormal} \ar{dl} \\
\tikst{integrally\\closed} \ar{dr}& \tikst{generalized\\Krull} \ar{l} \ar{r}& \tikst{perinormal} \ar{dl} \ar{dr} & \tikst{Pr\"ufer} \ar{l} \ar[crossout]{u}\\
& \tikst{weakly\\ normal} \ar{d} & & \tikst{\Rone} & \\
& \tikst{seminormal}
\end{tikzcd}
\]
\end{figure}

\section{Some subtleties of the non-Noetherian case}\label{sec:nonNoeth}

As usual, nuances exist for general commutative rings that do not come up when one assumes all rings involved are Noetherian.  We explore some of these in the current section.

\begin{example}\label{ex:Hutchins}
There is a non-Noetherian one-dimensional local integrally closed domain that isn't perinormal.  In fact, any integrally closed one-dimensional local domain that isn't a valuation domain will suffice. For example, let $K/F$ be a purely transcendental field extension, let $X$ be an analytic indeterminate over $K$, let $V := K[\![X]\!]$, and then let $R=F+XV$.  Then $R$ is easily seen to be local with maximal ideal $XV$ and integrally closed (but not completely integrally closed) in its fraction field $K(\!(X)\!)$.  To see that $R$ has dimension 1, let $\p \in \Spec R$ with $0 \subsetneq \p \subsetneq XV$.  Then by Lemma~\ref{lem:locms}, the map $R_\p \rightarrow V_{R \setminus \p}$ is an isomorphism, so that $V_{R\setminus \p}$ is local and hence equals either $V$ or $K(\!(X)\!)$.  The former is impossible since every nonunit of $R$ is a nonunit of $V$, and the latter means that $\p = 0$, which contradicts the assumption.  Hence, $\Spec R = \{0, XV\}$, whence $\dim R=1$.  But $V$ is a going-down local overring of $R$ that is not a localization.
\end{example}

\begin{example}
There exist non-Pr\"ufer, integrally closed integral domains (necessarily non-Noetherian) that are perinormal but not generalized Krull.  

For a concrete example, let $k$ be a field,  $x,y$  indeterminates over $k$, and  $R=k[x,y,\frac{y}{x},\frac{y}{x^2},\frac{y}{x^3},\ldots]$, considered as a subring of $k(x,y)$.   If $\m =xR$, then $\m$ is a maximal ideal of $R$ of height two (see below).  If $P$ is any other maximal  ideal of $R$, then $x\not\in P$.  Thus $R\subseteq k[x,y]_\p$, where $\p = P\cap k[x,y]$. Hence $R_P = k[x,y]_\p$, and so in particular $R_P$ is a Krull domain, whence perinormal.  It is also now clear that $R$ is not Pr\"ufer.  On the other hand, it is known (though apparently not written down) that $R_\m$ is a valuation ring.  Specifically, $R_\m$ is the valuation ring associated to the valuation $\nu$ on $k(x,y)$ with value group $\mathbb{Z}\times \mathbb{Z}$ (ordered lexicographically), where $\nu(x)= (0,1)$ and $\nu(y) = (1,0)$.
We give a brief outline as to why $R_\m=V$, where $V$ is the valuation ring of $\nu$.

Clearly $R\subset V$, since $y/x\in V$.   Moreover the maximal ideal of $V$ is generated by $x$, and so $R_\m \subseteq V$.   For the reverse containment we can write an arbitrary element of $V$ as $f/g$, where $f,g \in k[x,y]$.   Evidently we can assume that $g$ is not divisible by $y$.  Thus $\nu(g) = (0,m)$ for some nonnegative integer $m$.  We can then write $g= x^mh(x)+yp(x,y)$, where $h(x)\in k[x]$ and $p(x,y) \in k[x,y]$.  Since $\nu(f)\geq \nu(g)$,  $v(f) = (n,t)$ where either $n>0$ or $n=0$ and $t\geq m$.   In either case one can show that $f/g$  can be written in the form $F/G$, where $F, G \in k[x,y]$ and $G \notin \m$.
    Thus $f/g \in R_\m$, showing the two rings are equal.

Hence, $R$ is integrally closed. Finally to complete the example we must know that $R$ is not a generalized Krull domain.  However, $\m=xR$ is a principal prime ideal of height two.  Thus $ x^{-1} \in (\bigcap_{\p \in \Spec^1(R)}R_\p)\setminus R$, contradicting the definition of a generalized Krull domain.
\end{example}






\begin{example}
Let $V$ be any rank 1 valuation ring and $n \in \N$.  Recall that generalized Krull domains are closed under finite polynomial extension \cite[Theorem 43.11(3)]{Gil-MIT}.  Thus, $V[X_1, \ldots, X_n]$ is a generalized Krull domain (since $V$ is obviously generalized Krull), hence perinormal (by Theorem~\ref{thm:Krperi}).  This provides a large class of examples of perinormal domains, of arbitrary Krull dimension, that are neither Krull nor Pr\"ufer, even locally.
\end{example}

\section{Questions}\label{sec:Q}
We close with an incomplete but intriguing list of questions suitable for further research on perinormality and global perinormality.

\begin{question}
Let $k$ be a field, and let $X,Y,Z,W$ be indeterminates over that field.   Is the normal hypersurface $R=k[X,Y,Z,W] / (XW-YZ)$ globally perinormal or not?  Note that its divisor class group is well-known to be infinite cyclic \cite[Proposition 14.8]{Fos-Cl}. If ``yes", this answer would mean that the converse to Theorem~\ref{thm:class}\emph{\ref{it:classglobal}} is false in dimension 3.  If ``no", this answer would provide evidence that the converse may be true.
\end{question}

\begin{question}
Let $R$ be a perinormal domain, $K$ its fraction field, $L$ a (finite?) extension field of $K$, and $S$ the integral closure of $R$ in $L$.  Is $S$ perinormal?

In the non-Noetherian case, this question is interesting even when we further stipulate that $R$ is integrally closed.
\end{question}

\begin{question}
Let $R$ be an integral domain and $X$ an indeterminate over $R$.  What can one say about the perinormality of $R$ in relation to the perinormality of $R[X]$?  Does one imply the other?
\end{question}

\begin{question}
Let $R$ be a Noetherian local domain whose completion $\hat R$ is also a domain.  If $R$ is perinormal, is $\hat R$ perinormal as well?  What about the converse?
\end{question}

\begin{question} Is every completely integrally closed domain perinormal? \end{question}

\section*{Acknowledgments}
We wish to thank David Dobbs, Tiberiu Dumitrescu, Alan Loper, Karl Schwede, and Dana Weston for interesting and useful conversations at various stages of the project.  We also wish to thank the referee for many useful comments and improvements, especially for Propositions~\ref{pr:periflat} and \ref{pr:gperiflat}, which are due to the referee.
\providecommand{\bysame}{\leavevmode\hbox to3em{\hrulefill}\thinspace}
\providecommand{\MR}{\relax\ifhmode\unskip\space\fi MR }
\providecommand{\MRhref}[2]{%
  \href{http://www.ams.org/mathscinet-getitem?mr=#1}{#2}
}
\providecommand{\href}[2]{#2}

\end{document}